\documentclass{amsart}

\usepackage{amsfonts}
\usepackage{color}
\usepackage{graphicx}
\usepackage[colorlinks=true,urlcolor=blue,
citecolor=red,linkcolor=blue,linktocpage,pdfpagelabels,
bookmarksnumbered,bookmarksopen]{hyperref}
\usepackage{hyperref}
\usepackage[hyperpageref]{backref}

\newtheorem{thm}{Theorem}[section]
\newtheorem{cor}[thm]{Corollary}
\newtheorem{lema}[thm]{Lemma}
\newtheorem{prop}[thm]{Proposition}
\theoremstyle{definition}
\newtheorem{defn}[thm]{Definition}
\theoremstyle{remark}

\newtheorem{rem}[thm]{Remark}
\numberwithin{equation}{section}
\newcommand{\R}{\mathbb R}
\newcommand{\N}{\mathbb N}

\newcommand{\J}{\mathcal{J}}

\newcommand{\Xrad}{\mathcal{X}_{\rm rad}(B)}

\newcommand{\LL}{\mathcal{L}_{s,p,\beta}}

\newcommand{\ve}{\varepsilon}

\newcommand{\lam}{\lambda}

\newcommand{\cd}{\rightharpoonup}

\begin{document}

\title[Local--Nonlocal H\'enon problems]{On the mixed local--nonlocal H\'{e}non equation}

\author[A.M.\,Salort]{Ariel M. Salort}
\author[E.\,Vecchi]{Eugenio Vecchi}

\address[A.M.\, Salort]{Instituto del C\'alculo Luis A. Santal\'o (IMAS), CONICET, Departamento de Matem\'atica, FCEN - Universidad de Buenos Aires
\hfill\break \indent Ciudad Universitaria, Pabell\'on I (1428) Av. Cantilo s/n., Buenos Aires, Argentina.}
\email[A. Salort]{asalort@dm.uba.ar}
\urladdr{http://mate.dm.uba.ar/~asalort}

\address[E.\,Vecchi]{Dipartimento di Matematica
\newline\indent Politecnico di Milano \newline\indent
Via Bonardi 9, 20133 Milano, Italy}
\email{eugenio.vecchi@polimi.it}



\begin{abstract}
In this paper we consider a H\'{e}non-type equation driven by a nonlinear operator obtained as a combination of a local and nonlocal term. We prove existence and non-existence akin to the classical result by Ni, and a stability result as the fractional parameter $s \to 1$.
\end{abstract}

\maketitle
\section{Introduction}

Given $\beta\in [0,1]$, a fractional parameter $s\in (0,1)$ and $p>1$ we consider the following mixed local--nonlocal elliptic operator
\begin{equation}\label{eq:Operator}
\LL u := (1-\beta)(-\Delta_p)u + \beta(- \Delta_p)^s u.
\end{equation}
It is obvious that $\LL$ boils down to the usual $p-$Laplacian operator when $\beta=0$ and to the pure fractional $p-$Laplacian when $\beta=1$.
In the linear case $p=2$, this operator has a probabilistic interpretation consisting in describing a discontinuous process where the local part provides the continuous part, while the nonlocal one represents the jump process. We postpone to the end of the Introduction a brief account of the existing literature.

\medskip

In this paper we consider the following Dirichlet boundary value problem
\begin{equation} \label{prob}
\left\{\begin{array}{rll}
\LL u  =|x|^\alpha u^{q-1} & \text{in $B$} & \\
 u>0  & \text{in $B$} & \\
 u= 0 & \text{on $\partial B$} & \text{if $\beta =0$}\\
u  = 0 & \text{in $\R^n \setminus B$} & \text{if $\beta \in (0,1]$,}
\end{array}\right.
\end{equation} 
where $B$ stands for the unit ball in $\R^n$, $n\geq 2$,  $\alpha>0$, $s\in (0,1)$, $p>1$; we assume that $n>sp$ if $\beta=1$, while $n>p$ if $\beta \in [0,1)$, and $p<q<p^*_\beta$, where 
\begin{equation*}
  p^*_\beta := \left\{ \begin{array}{rl}
  \frac{np}{n-p} & \textrm{ if } \beta \in [0,1)\\
  \frac{np}{n-sp} & \textrm{ if } \beta =1.
  \end{array}\right.
\end{equation*}
Here we are primarily interested in proving existence and non-existence of weak solutions, which we define in a classical variational way:
we consider the functional $\J\colon X_{\beta,rad}(B) \to\R$ associated to \eqref{prob}
\begin{equation} \label{func.J}
\J(u)=\frac{1-\beta}{p} \|\nabla u\|^{p}_{L^{p}(B)}+\frac{\beta}{p} [u]_{s,p}^p- \frac{1}{q} \int_B |x|^\alpha (u^+)^{q} \,dx,
\end{equation}
defined on a suitable space $X_{\beta,rad}(B)$, see Section \ref{sec.functionSpace}, where $[u]_{s,p}$ stands for the Gagliardo seminorm 
\begin{equation}
[u]_{s,p}:=K(n,s)\, \left(\iint_{\mathbb{R}^{2n}}\dfrac{|u(x)-u(y)|^p}{|x-y|^{n+2s}}\, dx \, dy\right)^{1/p},
\end{equation}
where $K(n,s)$ is a positive constant, depending only on $n$ and $s$, such that $[u]_{s,p}^p\to \|\nabla u\|_p^p$ as $s\to 1$, see e.g. \cite{bbm}.\\
Clearly, $\J$ is class $C^1$ and its Frech\'et derivative is given by the expression
\begin{equation*}
\begin{aligned}
\langle  \J'(u),v\rangle &= (1-\beta)\int_{B}\|\nabla u\|^{p-2}\langle \nabla u, \nabla v \rangle \, dx \\
&+ \beta K(n,s) \iint_{\R^{2n}} \frac{|u(x)-u(y)|^{p-2}(u(x)-u(y))(v(x)-v(y))}{|x-y|^{n+sp}}\,dx \, dy \\
&- \int_{B} |x|^\alpha (u^+)^{q-1}v\,dx
\end{aligned}
\end{equation*}
for all $v\in X_{\beta,rad}(B)$. Therefore, any critical point of $\J$ is a weak solution of \eqref{prob}.

\medskip 

Obviously, when $\beta=0$ and $p=2$, problem \eqref{prob} boils down to the classical H\'enon equation introduced in \cite{Henon} to model spherically symmetric stellar clusters. The literature related to this famous equation is huge and encompasses several interesting lines of research in Nonlinear Analysis, including existence of solutions, nonexistence, multiplicity and finer qualitative properties of solutions.
Here we do not aim at providing a complete and fully detailed list of references, but rather we limit ourselves in mentioning the papers which are closely related to the   content of this note. Our main interest is to extend a classical result in \cite{N}, where Ni noticed that the presence of the term $|x|^\alpha$ was modifying the problem enough to increase the range of powers of $u$ for which a solution exists, in this way presenting a quite different scenario with respect to the case $\alpha =0$. Indeed, he was able to show that in there is a weak solution for $2<q<\tfrac{2n +2\alpha}{n-2} = 2^{\ast} + \tfrac{2\alpha}{n-2}$, so going beyond the classical nonexistence threshold $2^{\ast}$ related to the critical Sobolev embedding. Interestingly, due to the method for proving such existence, Ni was also able to get radiality of such a solution. Indeed, that solution is of Mountain Pass type, exploiting the Radial Lemma of Strauss and the compactness of radial Sobolev functions. This fact is noteworthy because the term $|x|^\alpha$ prevents from applying the symmetry results due to Gidas-Ni-Nirenberg already in the pure local case, and so even more in the mixed case, where similar qualitative results have been recently proved in \cite{BDVVAcc}. Actually, already in the pure local case, in \cite{SSW} it was proved that there is a sort of {\it critical threshold} for the parameter $\alpha$ beyond which there exist non-radial ground states. \\

Following the result of Ni, there have been several extension of it to different operators. Let us briefly recall them.
The case of the $p-$Laplacian was treated in \cite{Na} for $p>1$ while the purely linear nonlocal case (i.e., $\beta=1$ and $p=2$) was addresses in \cite{SW} , where it was proved that for $1<2s<n$ and $q<\frac{2n +2\alpha}{n-2s}$, there exists a positive weak solution for \eqref{prob}. The critical case $q=2^*_{s,\alpha}$ was studied in \cite{BQ}. To the best of our knowledge, the pure nonlocal and nonlinear case (i.e. $\beta =1$ and $p\neq 2$) has not be covered so far: this is the main reason why we choose the operator in \eqref{eq:Operator}, so to get existence and non-existence of solutions for both the nonlocal nonlinear case as well as mixed local-nonlocal combinations.

\medskip

Before stating our results, let us introduce a further threshold  quantity which will play a major role in the following:
\begin{equation}
p^*_{\beta,\alpha} := \left\{ \begin{array}{rl}
\frac{np + \alpha p}{n-p} & \textrm{ if } \beta \in [0,1)\\
\frac{np + \alpha p}{n-sp} & \textrm{ if } \beta =1.
\end{array}\right.
\end{equation}

Our first results concern the range of existence and non-existence of weak solutions to \eqref{prob}.

\begin{thm}\label{teo1} 
Let $\alpha>0$, $p>1$ and $s\in (0,1)$ be such that 
$$
p<n \quad \text{ when } \beta\in [0,1), \qquad sp<n \quad \text{ when } \beta=1.
$$
Then for all  $p<q<p^*_{\beta,\alpha}$ there exists a positive weak solution $u\in X_{\beta,rad}(B)$ to \eqref{prob}.

Moreover, we have that $u\in L^\infty(B)$ for $\beta\in [0,1)$ when $0<s<1<p<q<p^*_{\beta,\alpha}$ such are that $p<n$ and
\begin{equation} \label{bound.cond.1}
\alpha > \max\left\{0, (q-1)\left(\dfrac{n}{p}-1\right)-p\right\},
\end{equation}
and for $\beta=1$ when $0<s<1<p<q<p^*_{\beta,\alpha}$ are such   that $sp<n$ and
\begin{equation} \label{bound.cond.2}
 \alpha > \max\left\{0, (q-1)\left(\dfrac{n}{p}-s\right)-sp\right\},
\end{equation}
or for all $\alpha>0$ when   $1<p<q<p^*_{\beta,\alpha}$ are such that $sp<n$ and the range for $s$ is
\begin{equation} \label{bound.cond.3}
0<s< \frac{n}{p} \frac{q-1}{p+q-1}.
\end{equation}
\end{thm}

\medskip

\begin{thm} \label{teo2}
Let $\alpha>0$, $p>1$ and $s\in (0,1)$ be such that 
$$
p<n \quad \text{ when } \beta\in [0,1), \qquad sp<n \quad \text{ when } \beta=1.
$$
Then for all $q>p^*_{\beta,\alpha}$ problem \eqref{prob} has no solutions $u\in X_\beta(B)\cap W^{1,r}(B)\cap L^\infty(B)$, for some $r>1$.

In particular, there are no positive weak solutions  to \eqref{prob} for
\begin{itemize}
\item $\beta\in [0,1)$ when $p<n$ and \eqref{bound.cond.1} holds;
\item $\beta=1$ when  $sp<n$ and \eqref{bound.cond.2} holds, or when $\alpha>0$ and \eqref{bound.cond.3} holds.
\end{itemize}
\end{thm}

\medskip

The proof of Theorem \ref{teo1} follows the scheme introduced by Ni in \cite{N}. To this aim, we need several ingredients, like a proper version of the Radial Lemma (see Section \ref{sec:radialLemma}) and compactness of the embeddings of the appropriate radial Sobolev spaces. The proof of Theorem \ref{teo2} is heavily based on a previous non-existence result by Ros--Oton and Serra \cite{ROS}. In order to apply it, we have to prove a regularity result of independent interest, namely Theorem \ref{BOUNDED}, where we show the boundedness of weak solutions of
\begin{equation*}
    \begin{cases}
    \LL u = f & \text{in $\Omega$}, \\
    u = 0 & \text{in } \mathbb{R}^n\setminus \Omega,
   \end{cases}
   \end{equation*}
\noindent when $f \in L^{r}(\Omega)$ with $r> \tfrac{n}{p}$.

\medskip

Our next result establishes the stability of solutions of uniformly bounded solutions of \eqref{prob} as $s\nearrow 1$.

\begin{thm} \label{teo3}
Let $\beta \in (0,1]$. Let $p>1$ and $s_k \in (0,1)$ be such that $s_kp<n$ and $s_k \to 1$ as $k \to +\infty$. Given $p<q<p^*_{\beta,\alpha}$, let $u_{k}\in X_{\beta,rad}(B)$ be a solution of \eqref{prob} such that 
$$\sup_{k \in \mathbb{N}} \, \|u_k\|_{\beta} <+\infty.$$
Then, every accumulation point $u$ of $\{u_k\}_{k\in\N}$ in the $L^{p}(B)$-topology is a weak solution of \eqref{prob} with $\beta = 0$.
\end{thm}

Since the $\beta$-norm of weak solutions to \eqref{prob} can be bounded uniformly on $s$ (see Remark \ref{ToPassToTheLimit.2}), we immediately obtain the following consequence.

\begin{cor}
Let $\beta \in (0,1]$. Let $p>1$ and $s_k \in (0,1)$ be such that $s_kp<n$ and $s_k \to 1$ as $k \to +\infty$. Given $p<q<p^*_{\beta,\alpha}$, let $u_{k}\in X_{\beta,rad}(B)$ be a weak solution of \eqref{prob}. Then 
$$
u_k \to u \text{ strongly in } L^p(B) \text{ as } k\to\infty
$$
where $u \in X_{\beta,rad}(B)$ is a solution of \eqref{prob} with $\beta=0$.
\end{cor}

We remark that, since the above convergence implies in particular that $u_k\to u$ a.e., the limit function $u$ is actually radial. 

\medskip

Our last result extend the result of Smets, Willem and Su to the {\it linear} mixed case, showing the possibly intuitive fact that coupling a nonlocal part with the local one still allows for non-radial solutions to \eqref{prob}.

\begin{thm}\label{teo4}
Given $n\geq 3$ and $0\leq \beta < 1$, then for any $2+\beta<q<2^*$, there exists $\alpha^*>0$ such that no ground state of \eqref{prob} with $p=2$ is radial, provided $\alpha>\alpha^*$.
\end{thm}

We want to stress that the proof exploits several relations between the local and nonlocal terms and is heavily based on the original proof in \cite{SSW} in the linear case. This is the reason why the pure nonlocal case is actually not covered by our result.

\medskip

We want close this Introduction spending a few words regarding the existing literature: as far as we know, 
Fonduun \cite{Foondun} was the first who proved bounds for the heat kernel associated to \eqref{eq:Operator} for $p=2$, a Harnack inequality and a regularity theorem for {\it mixed local-nonlocal harmonic} functions. A few years later, in \cite{CKSV} the authors proved Harnack estimates in the linear case. Further regularity results for $p=2$ are contained in \cite{BDVV}, where the authors prove local $H^{k+2}$ estimates and several maximum principles. Qualitative properties of solutions of semilinear equations in the spirit of the classical results by Gidas, Ni and Nirenberg have been proved in \cite{BDVVAcc}, while a quantitative version of a Faber-Krahn inequality has been proved in \cite{BDVV2}. We refer also to \cite{GoelSreenadh} for related results. Recently, the nonlinear operator \eqref{eq:Operator} has also been considered: in \cite{BSM} the authors considered an eigenvalue problem for a system of local-nonlocal operators and studied the asymptotics as $p \to +\infty$. In a similar spirit, in \cite{SilvaSalort} the authors considered a nonlinear equation with concave-convex right hand side. We also quote \cite{BMV} where the authors considered a nonlinear version of the famous Brezis-Oswald problem, and \cite{GarainKinnunen,Garain3} where the authors proved regularity results and studied {\it mixed Sobolev inequalities} and quasilinear singular problems in the spirit of Boccardo and Murat. We finally mention \cite{DPLV1} for a first study of such operators with (nonlocal) Neumann boundary conditions. We want also to remind that there is a pretty active line of research dealing with evolution equations having the operator \eqref{eq:Operator} as elliptic part: we refer e.g. to \cite{DPLV2, GarainKinnunen2}  and to \cite{DV} where such operators are used to propose a model to describe the diffusion of a biological population living in an ecological niche and subject to both local and nonlocal dispersals. Finally, we want to mention that several papers, see e.g. \cite{CK,DPFR,dSOR,GQR}, have also investigated the coupling of a local operator with a nonlocal and nonsingular one having a different kernel with respect to the one considered in the present note.

\medskip

The paper is organized as follows: in Section \ref{sec.preliminaries} we collect all the preliminaries needed to prove our results. In particular, we introduce the functional and variational setting to study \eqref{prob}, we collect the different versions of the radial Lemma needed for our purposes and we prove the boundedness of the solutions. In Section \ref{sec.existence} we prove Theorem \ref{teo1} and Theorem \ref{teo2}, while the proof of Theorem \ref{teo3} is postponed to Section \ref{sec.stability}. In Section \ref{sec.nonradial} we prove Theorem \ref{teo4}.
Finally, we add a brief Appendix \ref{sec.app} where we recall the classical Mountain Pass Lemma and a few other classical results used along the paper.

\section{Preliminaries} \label{sec.preliminaries}
\subsection{Function spaces}\label{sec.functionSpace}
Problem \eqref{prob} is the local-nonlocal and nonlinear counterpart of the H\'enon problem studied in the celebrated paper by Ni \cite{N}.
As already mentioned in the Introduction, the problem has to be settled in the proper function space, which we consider to be
\begin{equation}
X_{\beta}(B):= \left\{ 
\begin{array}{rl}
W^{1,p}_{0}(B) & \textrm{if } \beta =0\\
X_p(B) & \textrm{if } \beta \in (0,1)\\
W^{s,p}_0(B) & \textrm{if } \beta =1,
\end{array}\right.
\end{equation}
\noindent where $W^{1,p}_{0}(B)$ is the classical Sobolev space whose functions have null trace on the boundary of $B$, while for $\beta =1$ we have the classical fractional Sobolev space
\begin{equation}
W^{s,p}_0(B) := \left\{ u \in W^{s,p}(\mathbb{R}^n): u = 0 \textrm{ in } \mathbb{R}^n\setminus B\right\}.
\end{equation}
Concerning the {\it true} mixed local-nonlocal case, i.e. for $\beta \in (0,1)$, the space $X_p(B)$ has been introduced in \cite{BMV} for more general open sets $\Omega \subset \mathbb{R}^n$ with $C^1$-smooth boundary. Nevertheless, to simplify the reading, we recall its precision definition in our specific case:
  \begin{equation} \label{eq.defSpaceXp}
   X_p(B) := \left\{u\in W^{1,p}(\mathbb{R}^n):\,
   \text{$u\equiv 0$ a.e.\,on $\mathbb{R}^n\setminus B$}\right\}.
  \end{equation}
  Since $\partial B$ is smooth, we can identify
  $X_p(B)$ with the space $W_0^{1,p}(B)$ in the following sense: denoting by $\mathbf{1}_B$ the indicator function of the ball $B$, we have that
  \begin{equation} \label{eq.identifXWzero}
   u \in W_0^{1,p}(B)\,\,\Longleftrightarrow\,\,
  u\cdot\mathbf{1}_{B}\in X_p(B).
  \end{equation}  
Therefore, in what follows  we will always identify  a function $u\in W_0^{1,p}(B)$ with 
  $\hat{u} := u\cdot\mathbf{1}_B \in X_p(B),$
  \noindent obtained as a zero-extension outside of $B$.
 \medskip
 We further notice that by the Poincar\'e inequality and \eqref{eq.identifXWzero}, we get that the quantity
  $$\|u\|_{X_p} :=\left( \int_{B}\|\nabla u\|^p\, dx\right)^{1/p},
  \qquad u\in X_p(B),$$ 
  endows $X_p(B)$ with a structure of real Banach space, which is
  actually isometric to $W_0^{1,p}(B)$. We briefly list a couple of useful properties that hold true:
  \medskip
  \begin{enumerate}
   \item[(i)] $X_p(B)$ is separable and reflexive (being $p > 1$);
   \vspace*{0.05cm}
   \item[(ii)] $C_0^\infty(B)$ is dense in $X_p(B)$.
  \end{enumerate}
  Summarizing, we will denote
\begin{equation}
\|u\|_{\beta} := \left\{ \begin{array}{rl}
\|\nabla u\|_{L^{p}(B)} & \textrm{ if } \beta \in [0,1)\\
\left[u\right]_{s,p} & \textrm{ if } \beta =1.
\end{array}\right.
\end{equation}

  Finally, since we are mainly interested in proving results related to radial functions, we introduce the following subspaces:
  \begin{equation}
  X_{\beta,rad}(B) := \{u\in X_{\beta}(B): u \textrm{ is radial}\}.
  \end{equation}

\subsection{The local-nonlocal operator}
Given $\beta\in [0,1]$, a fractional parameter $s\in (0,1)$ and $p>1$  the   mixed local--nonlocal elliptic operator
$$
\LL u := (1-\beta)(-\Delta_p)u + \beta(- \Delta_p)^s u
$$
is well defined between $X_{\beta}(B)$ and its dual space $X^*_{\beta}(B)$ and the following representation formula holds:
\begin{equation*}
\begin{aligned}
\langle \LL u,v\rangle  = (1-\beta) & \int_{B} |\nabla u |^{p-2}\langle \nabla u, \nabla v\rangle \, dx \\
&+ \beta K(n,s) \iint_{\R^{2n}} \frac{|u(x)-u(y)|^{p-2}(u(x)-u(y))(v(x)-v(y))}{|x-y|^{n+sp}}\,dx \, dy 
\end{aligned}
\end{equation*}
for any $v\in X_\beta(\R^n)$.

Therefore, we say that $u\in X_\beta(B)$ is a \emph{weak solution} of \eqref{prob} if
$$
\langle \LL u,v\rangle =\int_B |x|^\alpha u^{q-1}v\,dx \qquad \textrm{for all } v \in X_\beta(B).
$$
Observe that with this notation we have 
\begin{equation} \label{relation}
\langle \LL u,v\rangle = \langle \mathcal{J}'(u),v\rangle + \int_\Omega |x|^\alpha u^{q-1}v\,dx \quad \textrm{for all } v \in X_\beta(B).
\end{equation}

\begin{rem} \label{cumple.s.prop}
Given $u,v\in X_\beta(B)$, by H\"older's inequality
\begin{align*}
\langle \LL u,v\rangle  &\leq (1-\beta)\left(\int_B (|\nabla u|^{p-1})^{p'}\,dx\right)^\frac{1}{p'} \left( \int_B |\nabla v|^p\right)^\frac{1}{p}\\ 
&+ \beta  K(n,s) \left(\iint_{\R^{2n}} \left(\frac{|u(x)-u(y)|^{p-1}}{|x-y|^{\frac{n+sp}{p'}}} \right)^{p'} \,dxdy\right)^\frac{1}{p'}\left(\iint_{\R^{2n}} \frac{|v(x)-v(y)|^p}{|x-y|^{n+sp}}\,dxdy\right)^\frac{1}{p}\\
&=(1-\beta)\|\nabla u\|_{L^p(\R^n)}^{p-1} \|\nabla v\|_{L^p(\R^n)} +\beta  [u]_{s,p}^{p-1}[v]_{s,p}.
\end{align*}
This relation together with \cite[Proposition 2.2]{DRV} gives that
$$
\langle \LL u,v\rangle \leq  C \|u\|_\beta^{p-1} \|v\|_\beta   \quad \textrm{for all } u,v\in X_\beta(B) \text{ and for every } \beta \in [0,1].
$$

Therefore, in light of Proposition \ref{prop-s}, $\LL$ satisfies the so-called \emph{$({\bf S})$-property} of compactness (see Definition \ref{def.s}).
\end{rem}

\subsection{Radial Lemma}\label{sec:radialLemma}
In \cite{Strauss}, Strauss shed some light on the relation between the regularity and the decay of a Sobolev function in $H^{1}(\mathbb{R}^n)$. In particular, he proved a nowadays famous pointwise inequality, often referred to as {\it Strauss' Radial Lemma}, which reads as follows: let $n\geq 2$ and let $u \in H^{1}_{rad}(\mathbb{R}^n)$, then
\begin{equation}\label{eq:StraussLemma}
|u(x)| \leq C(n) |x|^{(1-n)/2} \left\{ \|u\|^{1/2}_{L^{2}(\mathbb{R}^n)}\|\nabla u\|^{1/2}_{L^{2}(\mathbb{R}^n)}\right\}, \quad \textrm{for a.e. } x \in \mathbb{R}^n.
\end{equation}
It is also possible to prove that there exists a radial function $\tilde{u}$ which coincides almost everywhere with $u$ and that it is continuous outside of the origin. The inequality \eqref{eq:StraussLemma} plays a crucial role in proving compact embeddings. According to \cite[Remark I.3]{Lions}, the above inequality holds replacing the exponent of $|x|$ with $-\alpha + \tfrac{2-n}{n}$ for every $\alpha \in \left[0,\tfrac{1}{2}\right]$. Now, the case $\alpha =0$ gives back the exponent found by Ni in \cite{N}, where he was dealing with radial functions in $H^{1}_{0}(B)$. 
The nonlinear case is treated in \cite{Lions}, where Lions proved the following: let $n \geq \max\{2,p\}$ and let $u \in H^{1,p}_{rad}(\mathbb{R}^n)$, then,
\begin{equation}\label{eq:LionsLemma}
|u(x)| \leq C(n,p) |x|^{-\gamma} \left\{ \|u\|^{(p-1)/p}_{L^{p}(\mathbb{R}^n)}\|\nabla u\|^{1/p}_{L^{p}(\mathbb{R}^n)}\right\}, \quad \textrm{for a.e. } x \in \mathbb{R}^n,
\end{equation}
\noindent  for every $\gamma = \tfrac{n}{p} - \delta$, where
$\delta \in \left[\tfrac{1}{p},1\right]$. The case $\delta=1$ gives the analogous of \cite{N} for a general $p\neq 2$.\\
For the nonlocal case, we refer to \cite{ChoOzawa} in the linear case (see also \cite{SW}), and \cite{SSV} for the case $p \neq 2$. \\


In order to simplify the readability, we list below the variants of the Strauss' Lemma needed for our purposes. We start from the pure nonlocal case.

\begin{lema} [\cite{SSV}]\label{strauss-ni}
Let $p>1$ with $1\leq sp<n$. Then there exists $C>0$ with
$$
|u(x)|\leq C |x|^{s-\frac{n}{p}}\|u\|_{W^{s,p}(\R^n)}
$$
for every $u\in W^{s,p}_{rad}(\R^n)$ and for a.e. $0<|x|\leq 1$.
\end{lema}

In the mixed case, the radial lemma can be stated as follows:
\begin{lema}\label{Lem:StraussMixed}
Let $p\geq 1$ and let $n \geq \max\{2,p\}$. Then, for every radial function $u \in X_{p,rad}(B)$, there exists a positive constant 
$C=C(n,p,\Omega)>0$ such that
\begin{equation}\label{eq:straussMixed}
|u(x)| \leq C |x|^{1-\frac{n}{p}} \|\nabla u\|_{L^{p}(B)}, \quad \textrm{for a.e. } 0<|x|\leq 1.
\end{equation}
\begin{proof}
We first notice that by the very definition of the (radial) Sobolev space $X_{p,rad}(B)$, and thanks to the regularity of $\partial B$, it actually holds that $u\in W^{1,p}(\mathbb{R}^n)$. Therefore, we can apply \cite[Lemma II.1, Remark II.3]{Lions}, finding that
\begin{equation}
|u(x)| \leq C(n,p)|x|^{(p-n)/p} \left\{ \|u\|^{(p-1)/p}_{L^{p}(\mathbb{R}^n)}\|\nabla u\|^{1/p}_{L^{p}(\mathbb{R}^n)}\right\}, \quad \textrm{for a.e. } x \in \mathbb{R}^n.
\end{equation}
In particular, the latter holds for a.e. $0<|x|\leq 1$. Now, it suffices to notice that
\begin{equation}
\begin{aligned}
\|u\|_{L^{p}(\mathbb{R}^n)} &= \|u\|_{L^{p}(B)}  \leq C(n,p,\Omega) \|\nabla u\|_{L^p(B)} \quad \textrm{(Poincar\'e ineq.)}\\
&\leq C(n,p,\Omega) \|\nabla u\|_{L^p(\mathbb{R}^n)}.
\end{aligned}
\end{equation}
Now, \eqref{eq:straussMixed} easily follows. 
\end{proof}
\end{lema}
It remains the pure local case corresponding to $\beta =0$. As for the considerations made on the space $X_p(B)$, thanks to the regularity of $\partial B$, we can {\it reverse} the identification of $X_p(B)$ with $W^{1,p}_0 (B)$ and use Lemma \ref{Lem:StraussMixed} also in the case $\beta = 0$.\\

We stress that the radial Strauss-Ni's Lemmas presented before are the key ingredients for the validity of the Palais-Smale condition.
	
\subsection{Boundedness}
We want now to discuss the boundedness of  the weak solutions of \eqref{prob}. In the case $\beta \in [0,1)$, we start with a {\em general theorem} of classical flavor whose proof is an adaptation of the classical method by Stampacchia. This approach has been already extended to the purely nonlocal setting, see e.g. the proof of Proposition~9	in~\cite{SV} and of Theorem 2.3 in~\cite{IRE}, and even to the {\it linear} mixed local-nonlocal case in \cite{BDVV}. 	Nevertheless, we will show all the details in order to make the paper self-contained. \\
	In the purely nonlocal case $\beta =1$, we refer to \cite[Theorem 3.1]{BP}. \medskip

    \begin{thm}\label{BOUNDED} 
     Assume that $n\geq 3$ and that $\beta \in [0,1)$.
	Let $f\in L^r(\Omega)$, with $1<p<n$, $r>n/p$, and 
 	assume that there exists the weak solution $u_f\in X_\beta (\Omega)$
	of
	\begin{equation} \label{eq.mainPB}
    \begin{cases}
    \LL u = f & \text{in $\Omega$}, \\
    u = 0 & \text{in } \mathbb{R}^n\setminus \Omega.
   \end{cases}
   \end{equation}
   	 Then, $u_f\in L^\infty(\mathbb{R}^n)$.

   \end{thm}
   
 	\begin{proof}[Proof of Theorem~\ref{BOUNDED}] 
 	We note first that the case $\beta = 0$ which corresponds to the purely local case is already well-known.\\
	Let $\delta>0$ be a positive number that we will conveniently choose later on.
	We can certainly assume that $u_f\not\equiv 0$ else there is nothing to prove. Therefore, 
    we can define the functions
	\begin{equation}\label{uASC} 
	\tilde u:=\frac{\delta^{1/p-1}\,{u}_f}
	{\|{u}_f\|_{L^{p^*}(\Omega)}+ \|f\|_{L^r(\Omega)}}
	\qquad{\mbox{and}}
	\qquad \tilde f:=\frac{\delta^{1/p-1}\,f}{\|{u}_f\|_{L^{p^*}(\Omega)}+\|f\|_{L^r(\Omega)}},
	\end{equation}
	\noindent which satisfy
	\begin{equation}\label{EQHS}
	\begin{cases}
\LL \tilde u=\tilde f & \text{in $\Omega$}, \\
	\tilde u=0 &  \text{in $\mathbb{R}^n\setminus\Omega$}.
	\end{cases}
	\end{equation}
	Now, for every $k\in\N$, we define the sequence of real numbers $C_k := 1-2^{-k}$ and the auxiliary functions
	$$v_k:=\tilde u-C_k, \quad w_k:=(v_k)_+:=\max\{v_k,0\},\quad
	U_k:= \|w_k\|_{L^{p^*}(\Omega)}^p.$$ 
	We notice that, by the Dominated
	Convergence Theorem,
	\begin{equation}\label{IJS:odf} 
	\lim_{k\to+\infty}U_k=
	\lim_{k\to+\infty}\|w_k\|^p_{L^{p^*}(\Omega)}=\|(\tilde u-1)_+\|^p_{L^{p^*}(\Omega)}.
	\end{equation}
	Also, if we take $k:=0$, we see that 
	$w_0=(v_0)_+ = (\tilde u-C_0)_+=\tilde u_+$. Now, denoting by $p^* :=\frac{np}{n-p}$ the classical Sobolev critical exponent, we get that
	\begin{equation}\label{DHAI:1}
	U_0=\left( \int_\Omega
	w_0^{p^*}(x)\, dx\right)^{p/p^*}
	\leq 
	\left( \int_\Omega \tilde u^{p^*}(x)\, dx\right)^{p/p^*}
	= \|\tilde u\|^p_{L^{p^*}(\Omega)}\le\delta^{p/p-1}.
	\end{equation}
	We will take it conveniently small in what follows.
	Now, since in $\mathbb{R}^n \setminus\Omega$ we have that 
	$v_{k+1}=-C_{k+1} \leq 0$ and thus
	$$w_{k+1}=0,$$
	we get that $w_{k+1}$ is an admissible test function, so that 
	\begin{equation}\label{EAFS1}
	\begin{split}
	& \int_\Omega |\nabla \tilde{u} |^{p-2} \langle \nabla\tilde u, \nabla w_{k+1}\rangle\, dx+
	\iint_{\mathbb{R}^2n}\!\!\!\!\!
	\frac{|\tilde u(x)-\tilde u(y)|^{p-2}(\tilde u(x)-\tilde u(y))(w_{k+1}(x)-w_{k+1}(y))}
	{|x-y|^{n+2s}}\, dx\, dy\\
	& \qquad = \int_\Omega w_{k+1}(x)\,\tilde f(x)\, dx.
	\end{split}
	\end{equation}
	We now exploit the fact that the nonlocal part has a positive sign to get rid of it: indeed, for a.e.\,$x,y\in \mathbb{R}^n$, we have (see, e.g.,~\cite[Lemma 10]{SV})
	\begin{equation}\label{EAFS2}
	 \begin{split} |w_{k+1}(x)-w_{k+1}(y)|^2
	 & = |(v_{k+1})_+(x)-(v_{k+1})_+(y)|^2 \\[0.2cm] 
	 & \leq ( (v_{k+1})_+(x)-(v_{k+1})_+(y))(v_{k+1}(x)-v_{k+1}(y)) \\[0.2cm]
	&= (w_{k+1}(x)-w_{k+1}(y))(\tilde u (x)-\tilde u(y)).
	\end{split}
	\end{equation}
 	Moreover,
	\begin{align*}
 	\int_\Omega |\nabla \tilde{u}(x)|^{p-2}\langle \nabla\tilde u(x), \nabla w_{k+1}(x)\rangle\, dx &=
	\int_{\Omega\cap\{ \tilde u>C_k\}}
	|\nabla \tilde{u}(x)|^{p-2}\langle \nabla\tilde u(x),\nabla v_{k+1}(x)\rangle \, dx \\
	&=	\int_\Omega |\nabla w_{k+1}(x)|^p\, dx.
	\end{align*}
	{F}rom this,~\eqref{EAFS1} and~\eqref{EAFS2} we can infer that
	\begin{equation*}
	\begin{split}&
	\int_\Omega |\nabla w_{k+1}(x)|^p\, dx
	\leq
	\int_\Omega w_{k+1}(x)\,\tilde f(x)\, dx.
	\end{split}
	\end{equation*}
	Hence, by Sobolev Inequality,
	\begin{equation}\label{Thsd-0}
	 U_{k+1} =
	\left(\int_\Omega|w_{k+1}(x)|^{p^*}\, dx\right)^{p/p^*}\leq 
	C\,\int_\Omega |\nabla w_{k+1}(x)|^p\, dx  \le C\,
	\int_\Omega w_{k+1}(x)\,|\tilde f(x)|\, dx,
	\end{equation}
	for some $C>0$. Also, $v_{k+1}\le v_k$ and therefore
	\begin{equation}\label{HA}
	w_{k+1}\leq w_k.
	\end{equation}
	Moreover, we observe that
	$$w_k=(\tilde u-C_k)_+ =\left(\tilde u-C_{k+1}+
	\frac1{2^{k+1}}\right)_+=
	\left(v_{k+1}+\frac1{2^{k+1}}\right)_+,$$
	and, as a result,
	\begin{equation}\label{HHA} 
	\{ w_{k+1}>0\} = \{ v_{k+1}>0\} 
	\subseteq \left\{w_k>\frac1{2^{k+1}}\right\}.
	\end{equation}
	We now introduce the following number
	\begin{equation}\label{deq} 
	 \tau := p^*\,\left(p^*-\frac{p^*}{r}-1\right)^{-1}<\frac{p^*}{p-1},
	 \end{equation}
	 \noindent in such a way that
	\begin{equation}\label{EXAO} 
	 \frac1{p^*}+\frac1r+\frac1\tau=1.
	\end{equation}
	Thanks to the lower bound for $q$, we also have that
	$$ p^*-\frac{p^*}{r}-1 > p^*-\frac{p^*}{n/p}-1
	 =\frac{np}{n-p}-\frac{p^2}{n-p}-1= p-1,$$
\noindent and it clearly follows from its very definition that
	$$ \tau> \frac{p^*}{p^*-1}>1.$$
	{F}rom this,~\eqref{HA} and~\eqref{HHA}, using H\"older inequality
	with exponents $p^*$, $r$ and $\tau$, we get that
	\begin{equation}\label{Thsd-1}
	\begin{split}
	 &\int_\Omega w_{k+1}(x)\,|\tilde f(x)|\, dx=
	\int_{\Omega\cap\{ w_{k+1}>0\}} w_{k+1}(x)\,|\tilde f(x)|\, dx \\[0.2cm] 
	&\qquad \le \|\tilde f\|_{L^{r}(\Omega)}\,\| w_{k+1}\|_{L^{p^*}(\Omega)}\,
	|\Omega\cap\{ w_{k+1}>0\}|^{1/\tau}\\[0.2cm] 
	&\qquad\le \|w_{k}\|_{L^{p^*}(\Omega)}\,
	\left|\Omega\cap\left\{w_k>\frac1{2^{k+1}}\right\}\right|^{1/\tau}\\[0.2cm] 
	&\qquad\le U_k^{1/p}\,\left( 2^{p^*(k+1)}
	\int_{\Omega\cap\left\{w_k>\frac1{2^{k+1}}\right\}} w_k^{p^*}\right)^{1/\tau} \\[0.2cm] 
	&\qquad\le \tilde C^k\,U_k^{1/p}\,U_k^{p^*/(p\tau)},
	\end{split}
	\end{equation}
	where $\tilde C>1$. 
	We now define
	$$\gamma:=\frac{1}{p} +\frac{p^*}{\tau p},$$
	and, thanks to~\eqref{deq}, we can easily notice that
	\begin{equation}
	\gamma>1,
	\end{equation}
	 By~\eqref{Thsd-0}
	and~\eqref{Thsd-1}, we can notice that
	$$U_{k+1}\le \hat C^k\;U_k^\gamma,$$
	for some $\hat C>1$.
	As a result, recalling~\eqref{DHAI:1} and keeping in mind that $\delta>0$ can be taken
	sufficiently small, we conclude that
	$$ \lim_{k\to+\infty} U_k=0.$$
	This and~\eqref{IJS:odf} give that
	$$\text{$\|(\tilde u-1)_+\|^p_{L^{p^*}(\Omega)}=0$},$$
	and therefore $\tilde u\leq 1$. As a consequence, recalling~\eqref{uASC}, for every $x\in\Omega$,
	\begin{equation}\label{3HY} 
	 {u}_f(x)\le \frac{\|{u}_f\|_{L^{p^*}(\Omega)}+\|f \|_{L^{r}(\Omega)}}{\delta},
	\end{equation}
	\noindent and this closes the proof.
\end{proof}

\begin{cor} \label{coro.bound}
Assume that $\beta \in [0,1)$, $0<s<1$, $1<p<n$ and let $u$ be a weak solution of \eqref{prob}. Then
\begin{equation}\label{eq:alphaBound}
u \in L^{\infty}(\mathbb{R}^n) \quad \textrm{for every } \alpha > \max\left\{0, (q-1)\left(\dfrac{n}{p}-1\right)-p\right\}.
\end{equation}

If $\beta=1$, $0<s<1$ is such that $1<sp<n$ and $u$ is a weak solution of \eqref{prob}, then
\begin{equation}\label{eq:alphaBound.1}
u \in L^{\infty}(\mathbb{R}^n) \quad \textrm{for every } \alpha > \max\left\{0, (q-1)\left(\dfrac{n}{p}-s\right)-sp\right\}.
\end{equation}

\begin{proof}
First we notice that if 
\begin{equation}\label{eq:estimater}
r \left( \alpha + \left(1 -\dfrac{n}{p}\right)(q-1)\right) +n >0,
\end{equation}
\noindent then 
$$|x|^{\alpha} u^{q-1} \in L^{r}(B).$$
This follows noticing that, thanks to Lemma \ref{Lem:StraussMixed}, it holds that
\begin{equation*}
\int_{B}|x|^{\alpha r} u^{r(q-1)}\, dx \leq C \int_{B}|x|^{\alpha r}|x|^{(q-1)r(p-n)/p}\, dx.
\end{equation*}
Now, it is clear that if
\begin{equation}\label{eq:alpha1}
\alpha \geq \left(\dfrac{n}{p}-1 \right)(q-1),
\end{equation}
\noindent then \eqref{eq:estimater} is trivially satisfied for every $r>0$, in particular for $r > \tfrac{n}{p}$. Therefore the conclusion follows from Theorem \ref{BOUNDED}.\\
On the other hand, if 
\begin{equation}\label{eq:alpha2}
\alpha < \left(\dfrac{n}{p}-1 \right)(q-1),
\end{equation}
it is enough to take
$$r = \dfrac{-n}{(q-1)\left(1-\dfrac{n}{p}\right) +\alpha},$$
\noindent and one can easily recognize that such $r > \tfrac{n}{p}$ {\it if and only if}
\begin{equation}\label{eq:alpha3}
\alpha > \left(\dfrac{n}{p}-1 \right)(q-1)-p.
\end{equation}
Combining \eqref{eq:alpha1}, \eqref{eq:alpha2} and \eqref{eq:alpha3} we immediately get \eqref{eq:alphaBound}.
An analogous argument by using Lemma \ref{strauss-ni} and \cite[Theorem 3.1]{BP} gives the result for $\beta=1$.

\end{proof}
\end{cor}

\begin{rem} \label{rem.bound}
It is clear that the maximum found in \eqref{eq:alphaBound} in the case $\beta\in [0,1)$ imposes a {\it new} relation between $p$, $q$ and $n$. By quite simple computations one can find bounds on $q$ in terms of $p$ and $n$ for which that maximum is actually $0$. This happens e.g. if $n=2,3,4,5$ and $q < \tfrac{p^2 +n -p}{n-p}$. For $n\geq 6$ one has more restrictions on $p$.
On the other hand, in the case $\beta=1$, we find that $u\in L^\infty(\R^n)$ for every $\alpha>0$ if we restrict the range of $0<s<\frac{n}{p}$ as
$$
0<s<  \frac{n}{p}\frac{q-1}{p+q-1}.
$$
\end{rem}

Finally, we recall the following well-known compactness result:

\begin{lema}\label{lemma.conv}
  Let $\Omega \subset \mathbb{R}^n$ be a bounded domain, let $p>1$ and let $\{u_n\}_{n\in\N}$ be a bounded sequence in
  $L^p(\Omega)$ such that $\{u_n\}_{n\in\N}$ converges to $u$
  a.e. Then $u\in L^p(\Omega)$ and $u_n\to u$ in
	$L^r(\Omega)$ for $r\in [1,p)$.  
\end{lema}

\section{Existence and non-existence results}\label{sec.existence}

The proof of the existence result follows by using the  mountain pass lemma (Proposition \ref{mountain}) due to Ambrosetti and Rabinowitz.

\begin{lema}\label{lema.comp}
Let $p>1$ and $\alpha>0$. Then the following holds:
\begin{itemize}
\item[(i)] \emph{(case $\beta=1$)} if $s\in (0,1)$ and $n > sp$, then
the inclusion $W^{s,p}_{rad}(\R^n)\subset L^r(|x|^\alpha,B)$
is compact provided that
  \begin{equation}\label{r-cond}
    1\le r <
    \begin{cases}
      \frac{np}{n-sp-\alpha p}& \text{if } \alpha<\frac{n-sp}{p}, \\[1\jot]
      \infty & \text{if } \alpha\ge \frac{n-sp}{p}.
    \end{cases}
  \end{equation}
  \item[(ii)] \emph{(case $\beta \in [0,1)$)} if $n \geq \max\{2,p\}$, then the inclusion $W^{1,p}_{rad}(\R^n)\subset L^r(|x|^\alpha,B)$
is compact provided that
  \begin{equation}\label{r-cond2}
    1\le r <
    \begin{cases}
      \frac{np}{n-p-\alpha p}& \text{if } \alpha<\frac{n-p}{p}, \\[1\jot]
      \infty & \text{if } \alpha\ge \frac{n-p}{p}.
    \end{cases}
  \end{equation}
\end{itemize}

\end{lema}
 
\begin{proof}
We start proving (i), hence let $u\in W^{s,p}_{rad}(\R^n)$. By virtue of  Lemma~\ref{strauss-ni}, we have 
\begin{align} \label{eqr1}
\begin{split}
\int_B |x|^{\alpha r} |u|^r \,dx &\leq c\|u\|_{W^{s,p}(\R^n)}^r \int_B |x|^{\alpha r + sr-\frac{rn}{p}}\,dx\\
&= 
n\omega_n c\|u\|_{W^{s,p}(\R^n)}^r \int_0^1 \rho^{r(\alpha  + s-\frac{n}{p})+n-1}\,d\rho\\
&=
n\omega_n c\|u\|_{W^{s,p}(\R^n)}^r \frac{1}{r(\alpha  + s-\frac{n}{p})+n}
\end{split}
\end{align}  
provided that $ n+ r(\alpha  + s-\frac{n}{p})>0$.

Now, let $\{u_n\}_{n\in\N}$ be a bounded sequence in $W^{s,p}_{rad}(\R^n)$. Up to a subsequence,  by the compact embedding theorem for fractional Sobolev spaces \cite[Corollary 7.2]{DRV}, 
\begin{align*}
&u_n \to u \text{ strongly in } L^m(B) \text{ for all } m<p^*_s\\
&u_n \to u \text{ a.e. in }B. 
\end{align*} 
Therefore,  
$$
|x|^\alpha u_n\to |x|^\alpha u  \text{ a.e. in } B.
$$
By estimate \eqref{eqr1} the sequence $\{|x|^\alpha u_n\}_{n\in\N}$ is bounded in $L^r(B)$ for all $r$ satisfying   \eqref{r-cond}.
In turn, in light  of Lemma \ref{lemma.conv}, it follows that  
$$
|x|^\alpha u_n \to  |x|^\alpha  u \text{ strongly in } L^q(B)
$$ 
for all $q <r$, and then the lemma follows.\\
The proof of (ii) runs in an analogous way, so we omit the details.
\end{proof}

Let us check that $\J$ satisfies the Palais-Smale condition.

\begin{lema} \label{l.e.2}
	Let $p<q<p^*_{\beta,\alpha}$. For $\beta \in [0,1]$ the functional $\J$ defined in \eqref{func.J} satisfies the conditions
  \begin{itemize}
  \item[(i)] $\J(0)=0$ and $\J(v) \le 0$ for some $v\ne 0$ in $X_{\beta,rad}(B)$,
  \item[(ii)] there exists $\mu\in \bigl(0,\|v\|_{\beta}\bigr)$ and
    $\sigma>0$ such that $\J\geq \sigma$ on
    $S_\mu := \{u\in X_{\beta,rad}(B) \, : \, \|u\|_{\beta} = \mu \}$.
  \end{itemize}
\end{lema}

\begin{proof}	
Let us check (i). Obviously $\J(0)=0$. Let $u_0\in X_{\beta,rad}(B)$ be a positive function such that $\|u_0\|_{\beta}=1$. Then, since $q>p$, we have that there exists a positive constant $C>0$ such that
\begin{align}\label{eq.forStability}
\J(tu_0)\le (1+C) \frac{t^p}{p} - \frac{t^q}{q}\int_{\Omega_\rho}  |x|^\alpha (u^+_0)^{q} \,d x  \leq 0
\end{align}	
for $t>0$ large enough. The constant $C$ can be taken equal to zero in the cases $\beta = 0$ or $\beta =1$.

Let us check (ii). First, we consider the pure nonlcal case $\beta =1$. We observe that $\J$ is well defined since for $x\in B$, by Lemma \ref{strauss-ni}
\begin{equation} \label{eqq.1}
\begin{split}
\int_B  |x|^\alpha (u^+)^{q}  \,dx &\leq c [u]_{s,p}^q \int_B |x|^{\alpha+sq-\frac{nq}{p}}\,dx\\
&=c [u]^q_{s,p}\int_0^1 r^{\alpha+q(s-n/p)+n-1}\,dr\\
&\leq c [u]^q_{s,p}
\end{split}
\end{equation}
for $q<p^*_{\beta,\alpha}$. 
Since $q>p$, from \eqref{eqq.1} we get  that
\begin{equation}\label{eqq.2}
\J(u)\geq \frac{1}{p}[u]_{s,p}^p-	c [u]_{s,p}^q =\frac{\mu^p}{p} - c \mu^q=\sigma>0
\end{equation}
if $\left[u\right]_{s,p}=\mu$ for $\mu>0$ small enough. \\
The case $\beta \in [0,1)$ works in a similar way replacing $[u]_{s,p}$ with $\|\nabla u\|_{L^{p}(B)}$ and taking $s=1$. This closes the proof.
\end{proof}

\begin{rem}\label{ToPassToTheLimit.2}
We note that the function $v$ in (i) is given by $tu_0$ in \eqref{eq.forStability}. Since $t$ can be taken {\it large enough} and $u_0$ has been normalized, we can infer that $\|t u_0\|_{\beta}$ does not depend on $s$ and therefore the Mountain Pass solution found in (ii) has a norm which is uniformly bounded (in $s$).
\end{rem}


\begin{lema}\label{2ps} 
Let $p<q<p^*_{\beta,\alpha}$.
 Then, the functional  $\J$ satisfies the Palais-Smale condition  for every $\beta\in [0,1]$
\end{lema}

\begin{proof}
Let us first prove that  every Palais-Smale sequence $\{u_n\}_{n\in\N}\subset X_{\beta,rad}(B)$ for $\J$ is   bounded. Since $\J'(u_n)\to 0$,
  \begin{equation} \label{ecl.1}
  \begin{aligned}
    |\langle \J'(u_n),u_n\rangle|
   & = \left| (1-\beta) \|\nabla u_n\|^{p}_{L^{p}(B)}+ \beta \left[u_n\right]_{s,p}^p
      - \int_{B} |x|^\alpha (u_n^+)^{q}\,d x \right|  \le \|u_n\|_{\beta},
      \end{aligned}
  \end{equation}
for $n$ large enough. The condition $|\J(u_n)|\leq C$ is equivalent to 
  \begin{equation} \label{ecl.2}
    \left|\frac{1-\beta}{p}\|\nabla u_n\|^{p}_{L^{p}(B)}+ \frac{\beta}{p} \left[u_n\right]_{s,p}^p
      -	\frac{1}{q}\int_{B} |x|^\alpha (u_n^+)^{q} \,d x \right|
    \le C.
  \end{equation}
From \eqref{ecl.1} and \eqref{ecl.2} it follows that
  \begin{align*}
   (1-\beta)\|\nabla u_n\|_{L^p(B)}^p + \beta [u_n]_{s,p}^p 
    &\le pC +\frac{p}{q}  \int_B |x|^\alpha (u_n^+)^q \, dx \\
    &\le pC  + \frac{p}{q}\|u_n\|_\beta + \frac{(1-\beta)p}{q} \|\nabla u_n\|_{L^p(\R^n)}^p  + \frac{\beta p}{q} \left[u_n\right]_{s,p}^p
  \end{align*}
from where, since $p<q$, we get that
  \begin{equation*}
    \left(1-\frac{p}{q}\right) ((1-\beta)\|\nabla u_n\|_{L^p(B)}^p + \beta [u_n]_{s,p}^p )
    \le pC + \frac{p}{q} \|u_n\|_{\beta}.
  \end{equation*}
In particular, when $\beta=1$ this gives that
 \begin{equation*}
    \left(1-\frac{p}{q}\right) [u_n]_{s,p}^p 
    \le pC + \frac{p}{q} [u_n]_{s,p},
  \end{equation*}  
and when $\beta \in [0,1)$, 
 \begin{equation*}
    \left(1-\frac{p}{q}\right)  \|\nabla u_n\|_{L^p(B)}^p
    \le pC + \frac{p}{q} \|\nabla u_n\|_{L^p(B)}.
  \end{equation*}  
Hence, for $\beta\in [0,1]$ the sequence $\{u_n\}_{n\in\N}$ is bounded in $X_{\beta,rad}(B)$.
Passing now to a subsequence, we have $u_n\rightharpoonup u$ in $X_{\beta,rad}(B)$. Let us first prove that
\begin{equation}
  \label{lcom}
  \lim_{n\to\infty}\int_B |x|^\alpha (u_n^+)^{q-1}(u_n-u) \,d x =0.
\end{equation}
Up to a further subsequence, the sequence $\{u_n-u\}_{n\in\N}$ converges weakly to zero in $L^{p^*_\beta}(B)$. Then, in order to get \eqref{lcom} we have to prove that $|x|^\alpha u_n^{q-1}\to |x|^\alpha u^{q-1}$ strongly in $L^{(p^*_\beta)'}$, where
\begin{equation}
\left(p^{*}_{\beta}\right)' = \left\{ \begin{array}{rl}
\frac{np}{np+p-n} & \textrm{ if } \beta \in [0,1),\\
\frac{np}{np+sp-n} & \textrm{ if } \beta = 1.
\end{array}\right.
\end{equation}  
 This is equivalent to prove that
$$
|x|^{\alpha/(q-1)} u_n  \to |x|^{\alpha/(q-1)}u  \text{ strongly in }
L^{(q-1)(p^{\ast}_{\beta})'}(B).
$$
By using Lemma \ref{lema.comp}, this is actually true whenever
\begin{equation}
(q-1)(p^{\ast}_{\beta})' < \left\{ \begin{array}{rl}
\frac{np}{n-p-\frac{\alpha p}{q-1}} & \textrm{ if } \beta \in [0,1),\\
\frac{np}{n-sp-\frac{\alpha p}{q-1}} & \textrm{ if } \beta =1,
\end{array}\right.
\end{equation}
that is, when
$$
q< p^*_{\beta,\alpha}.
$$

Therefore, from \eqref{relation}, we have for all $n\in\N$,
\begin{align*}
|\langle \LL(u_n),u_n-u\rangle|
&= \Bigl| \langle \J'(u_n),u_n-u\rangle
    + \int_B |x|^\alpha (u_n^+)^{q-1}(u_n-u) \,d x \Bigr| \\
  &\le \|\J'(u_n)\|_{X^*_{\beta}(B)}\|u_n-u\|_{\beta}
    +\int_B |x|^\alpha (u_n^+)^{q-1}(u_n-u)\,d x
\end{align*}
and the latter tends to $0$ as $n\to\infty$ by virtue of \eqref{lcom}. 

Finally, since by Remark \ref{cumple.s.prop}, $\LL$ fulfills the  $({\bf S})$-property of compactness, due to Proposition \ref{prop-s} the previous computations give that $u_n\to u$ strongly in $X_{\beta,rad}(B)$, which concludes the proof.
\end{proof}

We are now in position to prove our existence result.

\begin{proof}[Proof of Theorem \ref{teo1}]
From Lemmas \ref{l.e.2} and \ref{2ps} we are in position to apply the Mountain pass Theorem stated in Proposition \ref{mountain}. Indeed, when $1<p<q<p^*_{\beta,\alpha}$ there is a function $u\in X_{\beta,rad}(B)$ which is a critical point of $\mathcal{J}$ and hence is a non-trivial weak solution of \eqref{prob}. Moreover, Corollary \ref{coro.bound} and Remark \ref{rem.bound} ensure  that $u\in L^\infty(B)$ under our assumptions of the parameters.
\end{proof}

As already mentioned, the non-existence result follows from \cite{ROS}.

\begin{proof}[Proof of Theorem \ref{teo2}]
Given $u\in X_\beta(B)$, denote $u_\lam(x)=u(\lam x)$ for $\lam>1$. For any $\beta\in [0,1]$ it is easy to check that
$$
\|u_\lam\|_\beta    \leq \lam^{-\gamma}\|u\|_\beta
$$
with $\gamma=\frac{n-sp}{p}$ when $\beta\in [0,1)$ and $\gamma=\frac{n-p}{p}$ when $\beta=1$.

Denote $f(x,t)=|x|^\alpha t^{q-1}$ and $F(x,u)=\int_0^u f(x,t)\,dt$ and let $u\in X_\beta(B)$ be a weak solution of \eqref{prob}. A straightforward computation shows  that $f$ is supercritical in the sense that
$$
\beta t f(x,t)> n F(x,t)+x\cdot F_x(x,t) \qquad \text{for all }t\in B \text{ and } t\neq 0
$$
whenever $q > (n+\alpha)/\gamma$, i.e., when
\begin{equation} \label{cond.noexist}
q>\frac{p(n+\alpha)}{n-sp} \text{ when } \beta\in (0,1], \qquad q>\frac{p(n+\alpha)}{n-p} \text{ when } \beta=1.
\end{equation}

If $u\in X_\beta(B)\cap W^{1,r}(B)\cap L^\infty(\Omega)$, by \cite[Proposition 1.4]{ROS} we have that $u\equiv 0$.

In particular, if $u\in X_{\beta,rad}(B)$ is a weak solution of \eqref{prob}, by \cite[Proposition 2.2]{DRV} we have that $u\in W^{1,p}(B)$; moreover,  by Corollary \ref{coro.bound} and Remark \ref{rem.bound} $u\in L^\infty(B)$. As a consequence,  by \cite[Proposition 1.4]{ROS} we obtain that $u\equiv 0$.
\end{proof}

\section{Stability of solutions}\label{sec.stability}

This section is devoted to prove our stability result for solutions of \eqref{prob} as $s\nearrow 1$.

We start with the following useful lemma.

\begin{lema} \label{lemma.aux}
Let $s_k\uparrow 1$ and $v_k\in W^{1,p}_0(B)$ be such that $\sup_k \|\nabla v_k\|_{L^p(\R^n)}^p<\infty$. Assume without loss of generality that $v_k\to v$ strongly in $L^p(B)$. Then, for every $u\in W^{1,p}_0(\Omega)$ we have that
$$
\langle \mathcal{L}_{s_k,p,\beta} u,v_k \rangle \to  \langle -\Delta_p u,v_k \rangle .
$$
\end{lema}

\begin{proof}
If we prove that for any $u\in W^{1,p}_0(B)$
\begin{equation} \label{a.probar}
\langle -\Delta_p u,v \rangle \leq \liminf_{k\to\infty} \langle \mathcal{L}_{s_k,p,\beta} u,v \rangle,
\end{equation}
then applying \eqref{a.probar} to $-u$ gives the reverse inequality and hence the result.
By a refinement of \cite[Section 3]{bbm} we have
\begin{equation} \label{relacc1}
 \|\nabla(u+tv)\|_{L^p(\R^n)}^p \leq \liminf_{k\to\infty}  [u+tv_k]_{s_k,p}^p
\end{equation}
Moreover, from \cite{bbm} it follows that
\begin{equation}\label{relacc2}
\lim_{k\to\infty}    [u]_{s_k,p}^p = \|\nabla u\|_{L^p(\R^n)}^p.
\end{equation}

Denote $I:= \|\nabla (u+tv)\|_{L^p(\R^n)}^p - \|\nabla u \|_{L^p(\R^n)}^p$. We write
$$
I= \beta I + (1-\beta) I.
$$
By using the lower semicontinuity of the $L^p$ norm, we have
$$
(1-\beta) I \leq (1-\beta)\liminf_{k\to\infty} \left(\|\nabla (u+tv_k)\|_{L^p(\R^n)}^p - \|\nabla u \|_{L^p(\R^n)}^p \right).
$$
From \eqref{relacc1} and \eqref{relacc2} we get
$$
\beta I \leq \beta \liminf_{k\to\infty}\left( [u+tv_k]_{s_k,p}^p - [u]_{s_k,p}^p  \right).
$$
Then, from the last three relations and the superaditivity property of the liminf we obtain
$$
I \leq  \liminf_{k\to\infty} \left(   (1-\beta) \left( \|\nabla (u+tv_k)\|_{L^p(\R^n)}^p - \|\nabla u \|_{L^p(\R^n)}^p \right)  +  \beta \left( [u+tv_k]_{s_k,p}^p - [u]_{s_k,p}^p  \right) \right).
$$

Then, from \cite[Lemma 2.7]{fbs} it is immediate that
\begin{align*}
\langle -\Delta_p u,v \rangle + o(1) &\leq  \liminf_{k\to\infty} \left( (1-\beta)\langle - \Delta_p u,v_k \rangle + \beta \langle (-\Delta_p)^{s_k} u,v_k\rangle \right) + o(1)\\
&= \liminf_{k\to\infty} \langle \mathcal{L}_{s_k,p,\beta} u,v_k \rangle + o(1)
\end{align*}
from where \eqref{a.probar} follows.
\end{proof}

\begin{proof}[Proof of Theorem \ref{teo3}]
The proof closely follows the one of \cite[Theorem 3.3]{fbs}. 
We start assuming that $u_k \to u \quad \textrm{in } L^{p}(B)$. Since the sequence $\{u_k\}_{k\in \mathbb{N}}$ is uniformly bounded in $X_{\beta,rad}(B)$, and hence, due to \cite{bbm}, we can infer that $u \in W^{1,p}_0 (B)$. Possibly passing to a subsequence, we can also suppose that $u_k \to u$ a.e. in $B$. \\
Now, we define the sequence of functions $\{\eta_k\}_{k\in \mathbb{N}}$ as follows:
\begin{equation}
\eta_k := \mathcal{L}_{s_k,p,\beta} u_k \in W^{-1,p'}(B).
\end{equation}
\noindent By equation \eqref{prob} and the fact that $\sup_{k \in \mathbb{N}} \, \|u_k\|^{p}_{\beta}<\infty$, we get that the sequence $\{\eta_k\}_{k\in \mathbb{N}}$ is actually bounded in $W^{-1,p'}(B)$, and therefore, possibly passing once again to a subsequence, we can infer the existence of a function $\eta \in W^{-1,p'}(B)$ such that
\begin{equation}
\eta_k \rightharpoonup \eta \quad \textrm{weakly in } W^{-1,p'}(B).
\end{equation}

Now, since the $u_k$'s are weak solutions of \eqref{prob}, and exploiting the appropriate convergences, we find that for every $v \in W^{1,p}_{0}(B)$ it holds that
\begin{equation}
\begin{aligned}
0 &=  \langle \mathcal{L}_{s_k,p,\beta}, v \rangle - \int_{B}|x|^{\alpha}u_{k}^{q-1}v \, dx \to \langle \eta, v\rangle - \int_{B}|x|^{\alpha}u^{q-1}v \, dx \quad \textrm{as } k \to +\infty.
\end{aligned}
\end{equation}
The monotonicity of both the $p$-Laplacian and the fractional $p$-Laplacian implies that
\begin{equation}
\begin{aligned}
0 &\leq \langle \mathcal{L}_{s_k,p,\beta} u_k, u_k -v \rangle - \langle\mathcal{L}_{s_k,p,\beta} v, u_k -v \rangle \\
&= \int_{B}|x|^{\alpha}u_k^{q-1}(u_k -v) \, dx - \langle \mathcal{L}_{s_k,p,\beta} v, u_k -v \rangle.
\end{aligned}
\end{equation}
Passing to the limit as $k\to +\infty$, and using Lemma \ref{lemma.aux}, we get
\begin{equation}
\begin{aligned}
0 &\leq \int_{B}|x|^{\alpha}u^{q-1}(u-v) \, dx - \langle -\Delta_p v, u-v \rangle \\
&= \langle \eta, u-v \rangle - \langle -\Delta_p v, u-v \rangle.
\end{aligned}
\end{equation}

Now, taking $v = u -tw$ for a given $w \in W^{1,p}_{0}(B)$ and with $t>0$, we obtain that
\begin{equation}
\begin{aligned}
0&\leq \langle \eta, tw \rangle - \langle -\Delta_p (u-tw), tw \rangle\\
&= \langle \eta, w \rangle - \langle -\Delta_p (u-tw), w \rangle\\
&\to \langle \eta, w \rangle - \langle -\Delta_p u, w \rangle \quad \textrm{as } t \to 0^+.
\end{aligned}
\end{equation}
This shows that $u \in W^{1,p}_{0}(B)$ is a weak solution of \eqref{prob} with $\beta=0$. 
\end{proof}

\section{The linear case: existence of non-radial ground states}\label{sec.nonradial}


The aim of this section is, following the ideas of \cite{SSW},  to prove that under suitable conditions on $\beta$ and $q$, there exists $\alpha^*>0$ such that, ground states of $\LL$ in the linear case are no radial provided $\alpha>\alpha^*$.

We recall the following relation (see for instance \cite[Lemma 4.2]{FBS2}).
\begin{lema} \label{lemma.bound}
Let $u\in W^{1,p}(\R^n)$, $1\leq p < \infty$. Then, for each $0<s<1$
$$
\iint_{\mathbb{R}^{2n}}\dfrac{|u(x)-u(y)|^p}{|x-y|^{n+2s}}\, dxdy \leq \frac{n\omega_n}{p} \left( \frac{1}{1-s} \|\nabla u\|_{L^p(\Omega)}^p + \frac{2^p}{s} \|u\|_{L^p(\Omega)}^p \right).
$$
In particular, if $u\in X_p(B)$, for some $c=c(n,s,p)$
$$
\iint_{\mathbb{R}^{2n}}\dfrac{|u(x)-u(y)|^p}{|x-y|^{n+2s}}\, dxdy \leq  c\|\nabla u\|_p.
$$
\end{lema}

In order to simplify our notation, in this section we redefine the Gagiardo seminorm as
\begin{equation}
[u]_{s,p}:= \left( C(n,s,p)\iint_{\mathbb{R}^{2n}}\dfrac{|u(x)-u(y)|^p}{|x-y|^{n+2s}}\, dx \, dy\right)^{1/p},
\end{equation}
where, in light of Lemma \ref{lemma.bound} the constant $C$ is chosen such that
\begin{equation} \label{poinca}
[u]_{s,p}\leq \|\nabla u\|_{L^p(\Omega)} 
\end{equation}
for all $u \in X_p(B)$.

Moreover,  for $\beta\in [0,1)$ we set
$$
\mathcal{L}_{s,\beta} u := (-\Delta) u + \beta (- \Delta)^s u.
$$

We say that $u\in X_2(B)$ is a \emph{ground state} of 
$$
\mathcal{L}_{s,\beta} u =|x|^\alpha u^{q-1}, \; u>0 \text{ in } B
$$
if $u$ is a minimizer of 
$$
\inf_{0\neq u \in X_2(B)}  R(u)
$$
where
$$
R(u)=\frac{Z(u)}{N(u)} = \frac{  \beta [u]_{s,2}^2 + \|\nabla u\|_{L^2(\Omega)}^2 }{ \left(\int_B |x|^\alpha |u|^q\,dx\right)^\frac{2}{q} }.  	
$$

\begin{prop} 
Let $\alpha>0$, $0\leq \beta \leq 1$,  $n\geq 3$ and $q>2+\beta$. Then any radial minimizer of $R$ satisfies that
\begin{equation} \label{propo3.1}
\beta [u]_{s,2}^2 + \|\nabla u\|_{L^2(B)}^2 \leq  \frac{(1+\beta)(n-1)}{q-2-\beta} \int_B \frac{u^2}{|x|^2} \,dx
\end{equation}

\end{prop}

\begin{proof}
We assume that $N(u)=1$. Let $u\in X_2(B)$ be a minimizer of $R$ and define $g(t)= R(u+tv)$ for $t\geq0$ and $v\in X_2(B)$. Then $g''(0)\geq 0$ and $g'(0)=0$, which gives
$$
g''(0)=\frac{1}{N^2(u)}\left(\langle Z''(u)v,v\rangle N(u) - \langle N''(u) v,v\rangle Z(u) \right)
$$

An easy computation shows that
$$
\langle Z''(u)v,v\rangle =    2 \left(  \beta C \iint_{\R^n\times\R^n} \frac{ |v(x)-v(y)|^2}{|x-y|^{n+2s}}\,dxdy  +  \int_B |\nabla v|^2 \,dx \right)
$$

and that
$$
\langle N''(u)v,v\rangle = 2(2-q) \left( \int_B |x|^\alpha uv \,dx \right)^2 + 2(q-1) \int_B |x|^\alpha v^2\,dx.
$$

Then, $g''(0)\geq 0$ is equivalent to $I_1\leq I_2$ where
\begin{align*}
&I_1:=\left( (2-q) \left( \int_B |x|^\alpha uv \,dx \right)^2 + (q-1) \int_B |x|^\alpha |uv|^{2}\,dx \right) (   \beta  [u]_{s,2}^2 +  \|\nabla u\|_{L^2(B)}^2  ) \\
& I_2:=     \left( \beta C \iint_{\R^n\times\R^n} \frac{ |v(x)-v(y)|^2}{|x-y|^{n+2s}}\,dxdy  +\int_B  |\nabla v|^2 \,dx \right) 
\end{align*}

Observe that  by \eqref{poinca}
$$
I_2\leq  (1+\beta )      \|\nabla v\|_{L^2(B)}^2.
$$
 
Assume $u$ is radial and choose $v$ of the form $v=u(r)f(\sigma)$, where $f$ is a smooth function defined in $S^{n-1}$ with zero mean. Observe that $v\in X_2(B)$ for $n\geq  3$.

Since 
$$
|\nabla v|^2 = \left(\frac{\partial u}{\partial r}\right)^2 f^2 + \frac{u^2}{r^2} |\nabla_\sigma f|^2
$$
we have that
\begin{align*}
\frac{1}{1+\beta}I_2&\leq  \int_B |\nabla v|^2\,dx = \int_B |\nabla u(|x|)|^2 \,dx \int_{S^{n-1}} f^2\,dS_\sigma +   \int_B \frac{u(|x|)^2}{|x|^2} \,dx \int_{S^{n-1}} |\nabla_\sigma f|^2\,dS_\sigma\\
&\leq  \left(\beta   [u]_{s,2}^2 + \|\nabla u\|_{L^2(B)}^2 \right) \int_{S^{n-1}} f^2\,dS_\sigma +   \int_B \frac{u(|x|)^2}{|x|^2} \,dx \int_{S^{n-1}} |\nabla_\sigma f|^2\,dS_\sigma;
\end{align*}
since $f$ has zero mean and $N(u)=1$ we get
\begin{align*}
I_1&=\left( (2-q) \left( \int_B |x|^\alpha u(|x|)\,dx \int_{S^{n-1}} f\,dS_\sigma  \right)^2 + (q-1) \int_B |x|^\alpha  u(|x|)^2\,dx  \int_{S^{n-1}} f^2 \,dS_\sigma \right) Z(u)\\
&=(q-1) (  \beta   [u]_{s,2}^2 +  \|\nabla u\|_{L^2(B)}^2) \int_{S^{n-1}} f^2 \,dS_\sigma.
\end{align*}
Then $I_1\leq I_2$ means that
$$
(q-2-\beta) (   \beta  [u]_{s,2}^2 +  \|\nabla u\|_{L^2(B)}^2)  \int_{S^{n-1}} f^2  \,dS_\sigma  \leq   (1+\beta) \int_B \frac{u(|x|)^2}{|x|^2} \,dx \int_{S^{n-1}} |\nabla_\sigma f|^2\,dS_\sigma
$$
that is
$$
\frac{q-2-\beta}{1+\beta} (   \beta   [u]_{s,2}^2 + \|\nabla u\|_{L^2(B)}^2 )\leq  \mathcal{S}_n(f) \int_B \frac{u^2}{|x|^2} \,dx
$$
where 
$$
\mathcal{S}_n(f)=\frac{\int_{S^{n-1}} |\nabla_\sigma f|^2\,dS_\sigma}{\int_{S^{n-1}} f^2  \,dS_\sigma }.
$$
Since the infimum of $\mathcal{S}_n(f)$ over all $f\in H^1(S^{n-1})$ with $\int_{S^{n-1}}=0$ equals to $n-1$, we get
$$
\beta[u]_{s,2}^2 + \|\nabla u\|_{L^2(B)}^2 \leq  \frac{(1+\beta)(n-1)}{q-2-\beta} \int_B \frac{u^2}{|x|^2} \,dx
$$
and the proof concludes.
\end{proof}

We are now ready to prove Theorem \ref{teo4}.

\begin{proof}[Proof of Theorem \ref{teo4}]
Given $0<r\leq 1$, let $u_\alpha(r)\in X_2(B_r)$ be the ground state  of
$$
-\Delta u  = |x|^\alpha u^{q-1},\; u>0 \text{ in } B_r
$$
normalized such that $Z(u_\alpha(r))= 1$. When $r=1$ we just write   $u_\alpha$.

\emph{Step 1}. It holds that
\begin{equation} \label{step1}
\int_{B_r} |\nabla u_\alpha|^2 \,dx \to 0 \quad \text{ as } \alpha\to\infty.
\end{equation}
Indeed, from the equation for $u_\alpha$, as in \cite[Equation 16]{SSW}, we get
$$
\int_{B_r} |\nabla u_\alpha|^2\,dx \leq \int_{B_r} |\nabla u_\alpha|^2\,dx + \beta[u_\alpha]_{s,2}^2 \leq \frac{\int_{B_r} |x|^\alpha u_\alpha^q \,dx}{\int_{B} |x|^\alpha u_\alpha^q \,dx}.
$$
Now, since by \eqref{poinca}, $\|u_\alpha \|_{L^2(B_r)}^2 \leq \|u_\alpha \|_{L^2(B_r)}^2 + \beta[u_\alpha]_{s,2}^2 \leq (1+\beta)\|u_\alpha \|_{L^2(B_r)}^2$, the claim follows similarly as in the proof of \cite[Lemma 3.1]{SSW}.

\emph{Step 2}. Consider the ground state $u_\alpha$ in $B$. Let us see that
\begin{equation} \label{step2}
\int_B \frac{u_\alpha^2}{|x|^2}\,dx \to 0\quad \text{ as } \alpha\to\infty.
\end{equation}

Indeed, as in \cite[Proposition 3.1, Step 1]{SSW}, observe that there exists $0<r<1$ independent of $\alpha$ such that $u_\alpha(r)<\ve$.

 We decompose $B$ as 
\begin{align*}
\int_B \frac{u_\alpha^2}{|x|^2} \,dx&= 
\int_{B_r} \frac{u_\alpha^2}{|x|^2} \,dx + 
\int_{A_r} \frac{u_\alpha^2}{|x|^2} \,dx\\
&\leq 
2\int_{B_r} \frac{\tilde u_\alpha^2}{|x|^2} \,dx + 2\int_{B_r} \frac{u_\alpha(r)^2}{|x|^2} \,dx + 
\int_{A_r} \frac{u_\alpha^2}{|x|^2} \,dx 
\end{align*}
where $A_r=B\setminus B_r$,  $\tilde u_\alpha := u_\alpha - u_\alpha(r)$ and $\tilde u_\alpha \in X_2(B_r)$.

Observe that, as in \cite[Equation 13]{SSW}, from Hardy's inequality and \eqref{step1} we get
$$
\int_{B_r} \frac{\tilde u_\alpha^2}{|x|^2}\,dx  \leq \frac{4}{(n-2)^2}\int_{B_r} |\nabla u_\alpha|^2\,dx   \to 0 \quad \text{ as }\alpha\to\infty
$$
and, due to the election of $r$, 
$$
\int_{B_r} \frac{u_\alpha(r)^2}{|x|^2}\,dx \leq C\ve^2 \quad \text{ as } \alpha\to\infty.
$$
For the third term, observe that \eqref{step1} gives in particular that
$$
u_\alpha(r) \cd 0 \quad \text{ weakly in } X_2 \quad \text{ as } \alpha\to \infty.
$$
which, together with Rellich-Kondrakov theorem gives
$$
\int_{A_r} \frac{u_\alpha^2}{|x|^2}\,dx \to 0 \quad \text{ as } \alpha\to \infty.
$$
Mixing up the last expressions we obtain \eqref{step2}.

\medskip

\emph{Step 3}. Finally, from  \eqref{propo3.1} and \eqref{step2}  the result follows.
\end{proof}

\appendix
\section{Mountain pass lemma}\label{sec.app}

\begin{defn}
We say that the functional $\J$ satisfies the \emph{Palais-Smale compactness condition} if each sequence $\{u_n\}_{n\in\N}\subset \Xrad$ such that
\begin{itemize}
\item[(i)] $\{\J(u_n)\}_{n\in\N}$ is bounded, and
\item[(ii)] $\J'(u_n)\to 0$ in $\Xrad$
\end{itemize}
is precompact in $\Xrad$.
\end{defn}

We state the \emph{mountain-pass theorem} due to Ambrossetti  and Rabinowitz \cite{AR}.
\begin{prop} \label{mountain}
Let $E$ be a Banach space and let $\J\in C^1(E,\R)$ satisfy the Palais-Smale condition. Suppose that
\begin{itemize}
\item[(i)] $\J(0)=0$ and $\J(e)=0$ for some $e\neq 0$ in $E$;
\item[(ii)] there exists $\rho\in (0,\|e\|)$, $\sigma>0$ such that $\J\geq \sigma$ in $S_\rho=\{u\in E \colon \|u\|=\rho\}$.
\end{itemize}
Then $\J$ has a positive critical value
$$
c=\inf_{h\in \Gamma} \max_{t\in [0,1]} \J(h(t)) \geq \sigma>0
$$
where
$$
\Gamma=\{h\in C([0,1],E)\colon h(0)=0, h(1)=e\}.
$$
\end{prop}

\begin{defn} \label{def.s}
The functional $\J$ defined on $E$ satisfies the \emph{$({\bf S})$-property} if $\{u_n\}_{n\in\N}$ is a sequence in $E$ such that $u_n\rightharpoonup u$ weakly in $E$ and $\langle \J(u_n),u_n-u\rangle\to 0$, then $u_n\to u$ strongly in $E$.
\end{defn}

The following result characterizes the \emph{$({\bf S})$-property}. See \cite[Proposition 1.3]{PAO}.
 
\begin{prop} \label{prop-s}
Let $E$ be a uniformly convex Banach space and let $A_p\in C^1(E,\R)$ be such that
\begin{itemize}
\item[(i)] $\langle A_p(u),v \rangle\leq r \|u\|_E^{p-1}\|v\|_E$
\item[(ii)] $\langle A_p(u),u\rangle = r \|u\|_E^p$
\end{itemize}
for some $r>0$, for all $u,v\in E$. Then $A_p$ satisfies the \emph{$({\bf S})$-property}.
\end{prop}

\end{document}